\documentclass{article}
\usepackage{calc}
\usepackage{ifthen}
\usepackage{epsfig}
\usepackage{graphicx}
\usepackage{latexsym}
\usepackage{amssymb,amsmath,amsthm}
\usepackage{authblk}
\usepackage[hidelinks,
colorlinks=true,
linkcolor=myRed,
urlcolor=blue,
citecolor=myBlue
]{hyperref}
\usepackage{bookmark}
\usepackage{xcolor}
\usepackage{aliascnt}

\definecolor{myRed}{RGB}{200,0,0}
\definecolor{myBlue}{RGB}{0,0,200}

\newif\ifskip
\skiptrue
\newif\ifmargin
\marginfalse
\newtheorem{thm}{Theorem}{\bfseries}{\itshape}
\newaliascnt{lemm}{thm}% alias counter "<newTh>"
\newtheorem{lemm}[lemm]{Lemma}{\bfseries}{\itshape}
\aliascntresetthe{lemm}	
\newaliascnt{prop}{thm}% alias counter "<newTh>"
\newtheorem{prop}[prop]{Proposition}{\bfseries}{\itshape}
\aliascntresetthe{prop}	
\newtheorem{defi}[thm]{Definition}{\bfseries}{\itshape}
\newtheorem{examples}[thm]{Examples}{\bfseries}{\itshape}
{\bfseries}{\itshape}
\newtheorem{coro}[thm]{Corollary}{\bfseries}{\itshape}
{\bfseries}{\itshape}
{\bfseries}{\itshape}
{\bfseries}{\itshape}
\newtheorem{oproblem}{Problem}{\bfseries}{\itshape}
\renewcommand{\bar}{\overline}
\newcommand{\SOL}{\mathrm{SOL}}
\newcommand{\FOL}{\mathrm{FOL}}
\newcommand{\MSOL}{\mathrm{MSOL}}
\newcommand{\CMSOL}{\mathrm{CMSOL}}
\newcommand{\CFOL}{\mathrm{CFOL}}

\newcommand{\FPL}{\mathrm{FPL}}

\newcommand{\N}{{\mathbb N}}

\newcommand{\bZ}{{\mathbb Z}}

\newcommand{\fA}{{\mathfrak A}}
\newcommand{\fB}{{\mathfrak B}}
\newcommand{\fC}{{\mathfrak C}}

\newcommand{\cL}{\mathcal{L}}
\newcommand{\Str}{\mathrm{Str}}

\begin{document}
\title{Logics of Finite Hankel Rank
\\
\vskip1em
\small (For Yuri Gurevich at his 75th Birthday)
}
\author[1,2]{\normalsize Nadia Labai\thanks{Supported by the National Research Network RiSE (S114), and the LogiCS doctoral program (W1255) funded by the Austrian Science Fund (FWF).}}
\author[2]{\normalsize Johann A.~Makowsky\thanks{Partially supported by a grant of Technion Research Authority. \\ The final publication of this paper is available at Springer via \url{http://dx.doi.org/10.1007/978-3-319-23534-9_14}
}
}
\affil[1]{Department of Informatics, Vienna University of Technology \texttt{\small labai@forsyte.at}}
\affil[2]{Department of Computer Science, Technion - Israel Institute of Technology {\texttt{\small janos@cs.technion.ac.il}}}	

\renewcommand\Authands{ and }
\date{}
\maketitle
%%%%%%%%%%%%%%%%%%%%%%%%%%%%%%%%%%%%%%
\begin{abstract}
We discuss the Feferman-Vaught Theorem in the setting
of abstract model theory for finite structures.
We look at sum-like and product-like binary operations
on finite structures and their Hankel matrices. 
We show the connection between Hankel matrices and the 
Feferman-Vaught Theorem. The largest logic known to
satisfy a Feferman-Vaught Theorem for product-like operations is
$\CFOL$, first order logic with modular counting quantifiers.
For sum-like operations it is $\CMSOL$, the corresponding monadic 
second order logic.
We discuss whether there are maximal logics satisfying Feferman-Vaught Theorems
for finite structures.

\end{abstract}
%%\end{titlepage}
%\setcounter{tocdepth}{1}
%\tableofcontents 
%\newpage
%-----------------------------------------------------------------------------

\sloppy
\section{Introduction}
\subsection{Yuri's Quest for Logics for Computer Science}
The second author (JAM) first met Yuri Gurevich in spring 1976, while being a Lady Davis
fellow at the Hebrew University, on leave from the Free University, Berlin.
Yuri had just recently emigrated to Israel.
Yuri was puzzled by the supposed leftist views of JAM, %his counterpart, 
perceiving them as antagonizing. This lead to heated political discussions.
In the following time, JAM %The second author 
spent more visiting periods in Israel, culminating
in the Logic Year of 1980/81 at the Einstein Institute of the Hebrew University,
after which %leading finally to him joining 
he finally joined
the Computer Science Department at the Technion in Haifa.

At this time both Yuri and JAM worked on chapters to be published in \cite{bk:BF},
Yuri on Monadic Second Order Logic, and JAM on abstract model theory.
Abstract model theory deals with meta-mathematical characterizations of logic.
Pioneered by P. Lindstr\"om, G. Kreisel and J. Barwise, in
\cite{bk:BF,ar:Barwise74},\cite{ar:Kreisel68},\cite{ar:Lindstr1,ar:Lindstr2},
First Order Logic and admissible fragments of infinitary logic were characterized.
Inspired by H. Scholz's problem,
\cite{durand2012fifty},
R. Fagin initiated similar characterizations when models are restricted to finite models,
connecting finite model theory to complexity theory.

At about the same time Yuri and JAM
both underwent a transition in research orientation, slowly refocusing 
on questions in theoretical Computer Science.
Two papers document their
evolving views at the time, 
\cite{bk:gurevich},\cite{pr:JAMmth1}.
Yuri was vividly interested in 
\cite{pr:JAMmth1} and
frequent discussions between Yuri and JAM between 1980 and 1982 shaped both papers.
In \cite{pr:JAMmth1} the use  for theoretical computer science of classical model theoretic methods, 
in particular, the role of the classical preservation theorems (see below),
was explored, see also  \cite{ar:MMahr},\cite{ar:JAMhorn}.
Yuri grasped early on that these preservation theorems do not hold when one restricts
First Order Logic to finite models.

Under the influence of JAM's work in abstract model theory, the foundations
of database theory and logic programming,
\cite{pr:ChandraLewisM,pr:ChandraLewisM80},\cite{ar:MMahr},\cite{ar:JAMhorn},\cite{bk:BFxviii},\cite{ar:MVardi},
and the work of N. Immerman and M. Vardi,
\cite{ar:IMM1},\cite{pr:vardi82}, 
Yuri  stressed
the difference between classical model theory and finite model theory.
In \cite{bk:gurevich}, he
formulated what he calls the Fundamental Problem
of finite model theory. This problem is, even after 30 years, still open (\cite{bk:gurevich}):
Is there a logic $\mathcal{L}$ such that any class  $\Phi$ of finite structures is definable
in $\mathcal{L}$ iff $\Phi$ is recognizable in polynomial time.
For {\em ordered} finite structures there are several such logics,
\cite{Graedel91b},\cite{ar:IMM1},\cite{pr:JAM-kgc97},\cite{ar:MPlcc94,ar:MPcsl93},\cite{pr:vardi82}. 
We give a precise statement of the Fundamental Problem  in \autoref{se:qrank}, \autoref{FundProb}.

\subsection{Preservation Theorems}
Let $\mathcal{F}_1, \mathcal{F}_2$ be two syntactically defined fragments of a logic $\mathcal{L}$, and
let $R$ be a binary relation between structures.
{\em Preservation theorems} are of the form: 
\begin{quote}
Let $\phi \in \mathcal{F}_1$. 
The following are equivalent:
\begin{enumerate}
\item
For all structures 
$\mathfrak{A}, \mathfrak{B}$ 
with $R(\mathfrak{A}, \mathfrak{B})$, we have that if
$\mathfrak{A}$ satisfies $\phi_1 \in \mathcal{F}_1$, then also
$\mathfrak{B}$ satisfies $\phi_1$. 
\item
There is 
$\phi_2 \in \mathcal{F}_2$ which is logically equivalent to $\phi_1$.
\end{enumerate}
\end{quote}
A typical example is Tarski's Theorem for first order logic, with
$\mathcal{F}_1$ all of first order logic, 
$\mathcal{F}_2$ its universal formulas, and
$R(\mathfrak{A}, \mathfrak{B})$  holds if 
$\mathfrak{B}$  is a substructure of
$\mathfrak{A}$. Many other preservation theorems can be found in
\cite{bk:CK}. In response to \cite{pr:JAMmth1},\cite{ar:MVardi},
Yuri pointed out in
\cite{bk:gurevich} that most of the preservation theorems for first order logic
fail when one restricts models to be finite.

\subsection{Reduction Theorems}
Let $(\mathcal{F}_2)^*$ denote the finite sequences of formulas in $\mathcal{F}_2$, and
let $\Box$ be a binary operation on finite structures.
{\em Reduction theorems} are of the form: 
\begin{quote}
There is a function $p: \mathcal{F}_1 \rightarrow (\mathcal{F}_2)^*$
with $p(\phi)= (\psi_1, \ldots, \psi_{2 \cdot k(\phi)})$
and a Boolean function $B_\phi$
such that for all structures 
$\mathfrak{A}= \mathfrak{B}_1 \Box \mathfrak{B}_2$
and
all $\phi \in \mathcal{F}_1$, the structure
$\mathfrak{A}$ satisfies $\phi \in \mathcal{F}_1$
iff 
\begin{equation}
B(
\psi_1^{{B}_1}, \ldots , \psi_{k(\phi)}^{{B}_1},
\psi_1^{{B}_2}, \ldots , \psi_{k(\phi)}^{{B}_2}
) 
=1
\end{equation}
\end{quote}
where for $1 \leq j \leq k$ we have $\psi_j^{B_1} =1$ iff $\mathfrak{B}_1 \models \psi_j$
and $\psi_j^{B_2} =1$ iff $\mathfrak{B}_2 \models \psi_j$.
There are also versions for $(n)$-ary operations $\Box$.

The most famous examples of such reduction theorems are the
Feferman-Vaught-type theorems,
\cite{ar:Feferman68,ar:Feferman74b,ar:FefermanKreisel66,ar:FV},\cite{ar:gurevich79},\cite{ar:MakowskyTARSKI}.
A simple case is Monadic Second Order Logic ($\MSOL$), where 
$\mathcal{F}_1=\mathcal{F}_2 = \MSOL$ and
$\mathfrak{A}$ is the disjoint union $\sqcup$ of $\mathfrak{B}_1$ and $\mathfrak{B}_2$.
Additionally it is required that the quantifier ranks of the formulas in
$p(\phi)$ do not exceed the quantifier rank of $\phi$.
In \cite[Chapter 4]{bk:BFxviii} such reduction theorems are discussed in the context of
abstract model theory. However,  in \cite[Chapter 4]{bk:BFxviii} the quantifier rank has no role.

In contrast to preservation theorems, reduction theorems still hold when restricted
to finite structures.

\subsection{Purpose of this Paper}
In \cite{ar:MakowskyTARSKI} JAM discussed Feferman-Vaught-type theorems
in finite model theory and their  algorithmic uses.
In Section 7 of that paper, it was asked whether one can characterize logics over finite structures which
satisfy the Feferman-Vaught Theorem for the disjoint union $\sqcup$.
The purpose of this paper is to outline new directions to attack this problem.
The novelty in our approach is in relating the Feferman-Vaught Theorem to Hankel matrices 
of sum-like and connection-like
operations on finite structures.
Hankel matrices for connection-like operations, aka connection matrices, have many algorithmic applications,
cf. \cite{ar:LabaiMakowskyFPSAC},\cite{bk:Lovasz-hom}. 

In \autoref{se:qrank} we set up the necessary background on  Lindstr\"om logics, 
quantifier rank, translation schemes,
and sum-like operations.
A Hankel matrix  $H(\Phi, \Box)$ involves a binary operation $\Box$ on finite $\sigma$-structures which results in a $\tau$-structure, 
and a class of $\tau$-structures $\Phi$
closed under isomorphisms (aka a $\tau$-property).
In \autoref{se:hankel} we give the necessary definitions of Hankel matrices and their rank.
We then study $\tau$-properties $\Phi$ where $H(\Phi, \Box)$ has finite rank.
We show that there are uncountably many such properties and state that the class of all properties that have
finite rank for every sum-like operation $\Box$ forms a Lindstr\"om logic,
Theorems \ref{th:1} and \ref{th:nadia}.
In   \autoref{se:fv} we define various forms of Feferman-Vaught-type properties 
of Lindstr\"om logics equipped with
a quantifier rank, and discuss their connection to Hankel matrices.
\autoref{th:fv-main} describes their exact relationship.
A logic has finite S-rank, if all its definable $\tau$-properties have Hankel matrices 
of finite rank for every sum-like operation.
In  \autoref{se:sclosure} we sketch how to construct a logic satisfying the Feferman-Vaught Theorem 
for sum-like operations
from a logic which has finite S-rank.
Finally, in  \autoref{se:conclu}, we discuss our conclusions and state open problems.
A full version of this paper is in preparation, \cite{up:LabaiMakowskyFV}.

\section{Background}
\label{se:qrank}
\subsection{Logics with Quantifier Rank}
We assume the reader is familiar with the basic definitions of generalized logics,
see \cite{bk:BF},\cite{bk:EFT94}.
We denote by $\tau$ finite relational vocabularies, possibly with  constant symbols for named elements.
$\tau$-structures are always {\em finite} 
unless otherwise stated.
A finite structure of size $n$ is always assumed to have as its universe the set $[n]=\{ 1, \ldots , n\}$. 
A class of finite $\tau$-structures $\Phi$ closed under $\tau$-isomorphisms
is called a {\em $\tau$-property}.

A {\em Lindstr\"om Logic $\mathcal{L}$}  is a triple
$$
\langle \cL(\tau), \Str(\tau), \models_{\cL},\rangle
$$
where  $\cL(\tau)$ is the set of $\tau$-sentences of $\cL$,
$\Str(\tau)$ are the {\em finite} $\tau$-structures, $\models_{\cL}$ is the satisfaction relation.
The satisfaction relation is a ternary relation between $\tau$-structures, assignments and formulas.
An {\em assignment} for variables in a $\tau$-structure $\fA$ is a function which assigns to each variable an element of
the universe of $\fA$.
We always assume that the logic contains all the atomic formulas with free variables, and
is closed under Boolean operations and first order quantifications.
A logic $\cL_0$ is a {\em sublogic} of a logic $\cL$ iff $\cL_0(\tau) \subseteq \cL(\tau)$ for all $\tau$
and the satisfaction relation of $\cL_0$ is the satisfaction relation induced by $\cL$.

A {\em Gurevich logic $\cL$} is a Lindstr\"om logic where additionally
the sets $\cL(\tau)$ are  uniformly computable.

A {\em Lindstr\"om logic $\mathcal{L}$ with a quantifier rank} is a 
quadruple
$$
\langle \cL(\tau), \Str(\tau), \models_{\cL}, \rho \rangle
$$
where additionally $\rho$ is a {\em quantifier rank function}.
A quantifier rank (q-rank) $\rho$ is a function $\rho: \cL(\tau) \rightarrow \N$ such that
\begin{enumerate}
\item
For atomic formulas $\phi$ the q-rank $\rho(\phi)=0$.
\item
Boolean operations and translations induced by translation schemes (see  \autoref{subs:sumlike}) with formulas of q-rank $0$ preserve  maximal q-rank.
\end{enumerate}
A quantifier rank $\rho$ is {\em nice} if additionally it satisfies the following:

\begin{enumerate}
\item[(iii)]
For finite $\tau$,
there are, 
up to logical equivalence, 
only finitely many  $\cL(\tau)$-formulas 
of fixed q-rank
with a fixed set of free variables.
\end{enumerate}
In the presence of (iii) we define
{\em Hintikka formulas} 
as maximally consistent $\cL(\tau)$-formulas of fixed q-rank.
A {\em nice logic $\cL$} is Lindstr\"om logic with a  nice quantifier rank $\rho$.
We note that in a nice logic, the only formulas $\phi$ of q-rank $\rho(\phi)=0$
are Boolean combinations of atomic formulas.

We denote by $\FOL$, $\MSOL$, $\SOL$, first order, monadic second order, and full second order logic,
respectively. 
All these logics are nice Gurevich logics with their natural quantifier rank, and they are sublogics of $\SOL$.

We denote by $\CFOL$, $\CMSOL$, first order and monadic second order logic
augmented by the modular counting quantifiers $D_{k,m} x. \phi(x)$ which say that there are
modulo $m$, exactly $k$ many elements satisfying $\phi$.
In the presence of the quantifier $D_{k,m}$ there are two definitions of the quantifier rank:
$\rho_1(D_{k,m} x. \phi(x)) = 1 + \rho_1(\phi)$ and
$\rho_2(D_{k,m} x. \phi(x)) = m + \rho_2(\phi)$.
Given any finite set of variables, for $\rho_1$ we have, up to logical equivalence, infinitely many formulas $\phi$ with $\rho_1(\phi)=1$,
whereas for $\rho_2$ there are only finitely many such formulas.
$\CFOL$ and $\CMSOL$ with the quantifier rank $\rho_2$ are nice Gurevich logics.
In the sequel we always use $\rho_2$ as the quantifier rank for $\CFOL$ and $\CMSOL$.

$\FPL$, fixed point logic, is also a Gurevich logic and a sublogic of $\SOL$.
However, {\em order invariant $\FPL$} is a sublogic of $\SOL$ which is not a Lindstr\"om logic.
The definable $\tau$-properties in order invariant $\FPL$ are exactly the $\tau$-properties
recognizable in polynomial time.
For $\FPL$
and order invariant $\FPL$ see 
\cite{Graedel91b},\cite{ar:IMM1},\cite{pr:JAM-kgc97},\cite{ar:MPlcc94,ar:MPcsl93},\cite{pr:vardi82}.

\begin{oproblem}[Y. Gurevich, \cite{bk:gurevich}]
\label{FundProb}
Is there a Gurevich logic $\cL$ such that
the $\cL$-definable $\tau$-properties 
are exactly the $\tau$-properties
recognizable in polynomial time.
\end{oproblem}

\subsection{Sum-Like and Product-Like Operations on  \\ $\tau$-structures}
\label{subs:sumlike}
The following definitions are taken from \cite{ar:MakowskyTARSKI}.
Let $\tau, \sigma$ be two relational vocabularies with $\tau =\langle R_1, \ldots , R_m \rangle$,
and denote by
$r(i)$ the arity of $R_i$.
A {\em $(\sigma-\tau)$ translation scheme $T$} is a sequence of $\cL(\sigma)$-formulas $(\phi; \phi_1, \ldots, \phi_m)$
where $\phi$ has $k$ free variables, and each $\phi_i$ has $k \cdot r(i)$ free variables.
In this paper we do not allow redefining equality, nor do we allow name changing of constants.

We associate with $T$ two mappings 
$T^{\star}: \Str(\sigma) \rightarrow \Str(\tau)$ 
and 
$T^{\sharp}: \cL(\tau) \rightarrow \cL(\sigma)$, 
the {\em transduction and translation induced by $T$}. The transduction of a $\sigma$-structure $\fA$ is the $\tau$-structure $T^\star(\fA)$ where the vocabulary is interpreted by the formulas given in the translation scheme. The translation of a $\tau$-formula is obtained by substituting atomic $\tau$-formulas with their definition through $\sigma$-formulas given by the translation scheme.
A translation scheme (induced transduction, induced translation) is {\em scalar} if $k=1$, otherwise it is {\em  $k$-vectorized}. It is {\em quantifier-free} if so are the formulas $\phi; \phi_1, \ldots, \phi_m$.

If $\tau$ has no constant symbols,
the disjoint union  $\fA \sqcup \fB$ of two $\tau$-structures $\fA, \fB$ is the $\tau$-structure obtained by taking the disjoint union of the universes
and of the corresponding relation interpretations in $\fA$ and $\fB$.
On the other hand, if $\tau$ has finitely many constant symbols $a_1, \ldots, a_k$ the disjoint union of two
$\tau$-structures is a $\tau'$-structure with twice as many constant symbols, $\tau' = \tau \cup \{a_1', \ldots, a_k'\}$.
{\em Connection operations} are similar to disjoint unions with constants, where equally named elements are identified.
We call the disjoint union followed by the pairwise identification of $k$  constant pairs the $k$-sum,
cf. \cite{bk:Lovasz-hom}.

A binary operation $\Box: \Str(\sigma) \times \Str(\sigma) \rightarrow \Str(\tau)$ is {\em sum-like (product-like)} 
if it is obtained from the disjoint union of $\sigma$-structures by
applying a quantifier-free scalar (vectorized) $(\sigma-\tau)$-transduction.
A binary operation $\Box: \Str(\sigma) \times \Str(\sigma) \rightarrow \Str(\tau)$ is {\em connection-like}  
if it is obtained from  a connection operation on $\sigma$-structures by
applying a quantifier-free scalar $(\sigma-\tau)$-transduction.
If $\sigma = \tau$, we say $\Box$ is an operation on 
$\tau$-structures.

Connection-like operations are not sum-like according to the definitions in this paper\footnote{They are nevertheless called sum-like
in  \cite{ar:MakowskyTARSKI}.}.
Although connection operations are frequently used in the literature, cf. \cite{bk:Lovasz-hom},\cite{ar:MakowskyTARSKI},
we do not deal with them in this paper. Most of our results here can be carried over to connection-like operations,
but the formalism required to deal with the identification of constants is tedious and needs more place than available here.

\begin{prop}
Let $\tau$ be a fixed finite relational vocabulary. 
\begin{enumerate}
\item
There are only finitely many  sum-like binary operations on $\tau$-structures.
\item
There is a function $\alpha: \N \rightarrow \N$ such that
for each $k \in \N$ there are only $\alpha(k)$ many $k$-vectorized product-like binary operations on $\tau$-structures.
\end{enumerate}
\end{prop}

\subsection{Abstract Lindstr\"om Logics}

In \cite{ar:Lindstr2} a syntax-free definition of a logic is given.
An  {\em abstract Lindstr\"om logic $\cL$} consists of a family $\mathrm{Mod}(\tau)$ of $\tau$-properties closed under
certain operations between properties of possibly different vocabularies.
One thinks of  $\mathrm{Mod}(\tau)$  as the family of $\cL$-definable $\tau$-properties.
We do not need all the details here, the reader may consult \cite{bk:BF},\cite{ar:Lindstr1,ar:Lindstr2}.
The main point we need is that every abstract Lindstr\"om logic $\cL$ can be given a canonical
syntax  $\cL(\tau)$ using generalized quantifiers.

\section{Hankel matrices of $\tau$-properties}
\label{se:hankel}

\subsection{Hankel Matrices}
In linear algebra, a {\em Hankel matrix}, named after Hermann Hankel, is a  real or complex 
square matrix with constant skew-diagonals.
In automata theory, a {\em Hankel matrix} $H(f,\circ)$ is an infinite matrix where the rows and columns are labeled with
words $w$ over a fixed alphabet $\Sigma$, and the entry $H(f,\circ)_{u,v}$ is given by $f(u \circ v)$.
Here 
$f: \Sigma^\star \rightarrow \mathbb{R}$ is a real-valued word function
and $\circ$ denotes concatenation. 
A classical result  of G.W.~Carlyle and  A.~Paz  \cite{ar:CarlylePaz1971} 
in automata theory characterizes real-valued word functions $f$
recognizable by weighted (aka multiplicity) automata in algebraic terms.

Hankel matrices 
for graph parameters  
(aka connection matrices) 
were introduced by
L. Lov\'asz \cite{ar:Lovasz07} and used in \cite{ar:FreedmanLovaszSchrijver07},\cite{bk:Lovasz-hom} 
to study real-valued partition functions of graphs.
In  \cite{ar:FreedmanLovaszSchrijver07},\cite{bk:Lovasz-hom}
the role of concatenation is played by $k$-connections
of $k$-graphs, i.e., graphs with $k$ distinguished vertices $v_1, \ldots, v_k$. 

In this paper we study  $(0,1)$-matrices which are Hankel matrices of properties of general relational $\tau$-structures
and the role of $k$-connections is played by more general binary operations,
the sum-like and product-like operations introduced in \cite{th:Ravve95} and further studied in
\cite{ar:MakowskyTARSKI}.

\begin{defi}
Let $\Box: \Str(\sigma) \times \Str(\sigma) \rightarrow \Str(\tau)$ 
be a binary operation on finite $\sigma$-structures returning a $\tau$-structure,
and let $\Phi$ be a $\tau$-property.
\begin{enumerate}
\item
The Boolean Hankel matrix $H(\Phi, \Box)$ is the infinite $(0,1)$-matrix
where the rows and columns are labeled by all the finite $\sigma$-structures,
and $H(\Phi, \Box)_{\fA,\fB} =1$ iff $\fA \Box \fB \in \Phi$.
\item
The rank of
$H(\Phi, \Box)$ over $\bZ_2$ is denoted by
$r(\Phi, \Box)$, and is referred to as the Boolean rank. 
\item
We say that $\Phi$ has finite $\Box$-rank iff
$r(\Phi, \Box)$  is finite.
\item
Two $\sigma$-structures are $(\Phi,\Box)$-equivalent, $\fA \equiv_{\Phi,\Box} \fB$,
if for all finite $\sigma$-structures $\fC$ we have
\begin{equation}
\fA \Box \fC \in \Phi \ \mbox{   iff   } \ \fB \Box \fC \in \Phi 
\end{equation}
\item
For a $\sigma$-structure $\fA$, we denote by $[\fA]_{\Phi,\Box}$ the
$(\Phi,\Box)$-equivalence class of $\fA$.
\item
We say that $\Phi$ has finite $\Box$-index\footnote{
K. Compton and
I. Gessel, \cite{ar:Compton83},\cite{ar:Gessel84}, already considered $\tau$-properties of 
finite $\sqcup$-index for the disjoint union of $\tau$-structures. 
In \cite{ar:FischerKotekMakowsky11} this is called {\em Gessel index}.
C. Blatter and E. Specker, in \cite{pr:BlatterSpecker81},\cite{ar:Specker88}, consider a substitution operation on
pointed $\tau$-structures, $Subst(\fA, a, \fB)$, where the structure $\fB$ is inserted into $\fA$ at a point $a$. 
$Subst(\fA, a, \fB)$ is sum-like, and  the
$Subst$-index is called in \cite{ar:FischerKotekMakowsky11} Specker index.
}
iff
there are only finitely many
$(\Phi,\Box)$-equivalence classes.
\end{enumerate}
\end{defi}

\begin{prop}
\label{prop:index}
Let $\Phi$ be a $\tau$-property.
\\
$\Phi$ has finite $\Box$-rank iff
$\Phi$ has finite $\Box$-index. 
\end{prop}
\begin{proof}[Sketch of proof]
We first note that two $\sigma$-structures $\fA, \fB$ are in the same equivalence class of $\equiv_{\Phi,\Box}$
iff they have identical rows in $H(\Phi,\Box)$.
As the rank is over $\bZ_2$, finite rank implies there are only finitely many different rows in $H(\Phi,\Box)$.
The converse is obvious. 
\end{proof}

\subsection{$\tau$-Properties of Finite $\Box$-rank}
We next show that there are uncountably many $\tau$-properties of finite $\sqcup$-rank.
We also  study the relationship between the $\Box_1$-rank and $\Box_2$-rank of $\tau$-properties for 
different operations $\Box_1$ and $\Box_2$.

We first need a lemma.
\begin{lemm}
\label{le:periodic}
Let $A \subseteq \N$  and let
$M_A$ be the infinite $(0,1)$-matrix whose columns and rows are labeled by the natural numbers $\N$,
and $(M_A)_{i,j}=1$ iff $i+j \in A$.
Then $M_A$ has finite rank over $\bZ_2$ iff $A$ is ultimately periodic.
\end{lemm}

\begin{thm}
\label{th:1}
Let $\tau_{graphs}$ be the vocabulary with one binary edge-relation, and $\tau_1$ be $\tau_{graphs}$ augmented by one vertex label.
Let $C_A, \bar{C_A}$ and $P_A$ be the graph properties defined by 
$C_A =\{K_n: n \in A\}$,
$\bar{C_A} =\{E_n: n \in A\}$,
 and
$P_A =\{P_n: n \in A\}$, where
$E_n$  is the complement graph of the clique $K_n$ of size $n$, and $P_n$ is a path graph of size $n$.
\begin{enumerate}
\item
$H(C_A, \sqcup)$ has finite rank for all $A \subseteq \N$.
\item
For two graphs $G_1, G_2$, let $G_1 \sqcup^c G_2$ be the sum-like operation defined as the loopless complement graph of $G_1 \sqcup G_2$.
\\
$H(C_A, \sqcup^c)$ has infinite rank for all 
$A \subseteq \N$ which are not ultimately periodic.
\\
Equivalently, for the $\tau_{graphs}$-property $\bar{C_A}$, the Hankel matrix $H(\bar{C_A}, \sqcup)$ 
has infinite rank for all $A \subseteq \N$ which are not ultimately periodic.
\item
$H(P_A,\Box)$ has finite rank for all sum-like operations $\Box$ on $\tau_{graphs}$-structures and all $A \subseteq \N$.
\item
For two graphs $G_1, G_2$ with one vertex label, 
i.e. $\tau_1$-structures, let $G_1 \sqcup^1 G_2$ be the sum-like operation defined as the graph resulting from $G_1 \sqcup G_2$
by adding an edge between the two labeled vertices and then
removing the labels.
$H(P_A,\sqcup^1)$ has infinite rank for all 
$A \subseteq \N$ which are not ultimately periodic.
\item
$H(C_A, \sqcup_k)$ has finite rank for all $A \subseteq \N$.
\end{enumerate}
\end{thm}

\autoref{th:1}  needs an interpretation:
(i) says that there is a specific sum-like operation $\Box$ such that there uncountably many classes of $\tau$-structures
with finite  $\Box$-rank\footnote{
A similar construction was first suggested by E. Specker in conversations with the second author in 2000,
cf. \cite[Section 7]{ar:MakowskyTARSKI}.}.
(ii) says that if a class has finite $\Box$-rank for one sum-like operation, it does not have to hold for
all sum-like operations\footnote{
This observation was suggested by T. Kotek in conversations with the second author in summer 2014.}.
(iii) produces uncountably many classes  of $\tau$-structures which have finite  $\Box$-rank for all sum-like operations
on $\tau$-structures.
(iv) finally shows that such classes can still have infinite $\Box$-rank for sum-like operations
which take as inputs $\sigma$-structures (labeled paths) and output a $\tau$-structure (unlabeled paths).
This leads us to the following definition:

\begin{defi}
Let $\tau$ be a vocabulary and $\Phi$ be a $\tau$-property.
\begin{enumerate}
\item
$\Phi$ has finite S-rank (P-rank, C-rank) if for every sum-like (product-like, connection-like) operation 
$\Box: \Str(\sigma) \times \Str(\sigma) \rightarrow \Str(\tau)$ 
the Boolean rank of $H(\Phi,\Box)$  is finite.
\item
A nice logic $\cL$ has finite S-rank (P-rank, C-rank) iff all its definable properties have finite S-rank (P-rank, C-rank).
\end{enumerate}
\end{defi}

\begin{examples}
\
\begin{enumerate}
\item
(\cite{ar:GodlinKotekMakowsky08}):
$\FOL$ and $\CFOL$
have finite S-rank, C-rank and P-rank.
\item
(\cite{ar:GodlinKotekMakowsky08}):
$\MSOL$ and $\CMSOL$ have finite S-rank and C-rank.
\item
The examples $C_A, P_A$ above do not have finite S-rank.
\end{enumerate}
\end{examples}

\subsection{Proof of \autoref{th:1}}
\begin{proof}
(i) The disjoint union of two graphs is never connected. Therefore all the entries of $H(C_A \sqcup)$ are zero,
unless we consider the empty graph to be structure. In this case we have exactly one row and one column representing $C_A$.
In any case, the rank is $\leq 2$.
\\
(ii) Consider the submatrix of $H(C_A, \sqcup)$ consisting of rows and columns labeled  with the
edgeless graphs $E_n$ and use \autoref{le:periodic}.
\\
(iii)
We first observe that 
\begin{quote}
(*)
for any sum-like operation $\Box$ on $\tau_{graphs}$-structures (i.e., graphs), $G$ and $H$,
if $G \Box H = P_n$ for $n \geq 3$, either $G$ or $H$ must be the empty graph.
\end{quote}
This is due to the fact that  $\tau_{graphs}$ has no constant symbols.
Therefore, a row or column containing non-zero entries must be labeled by the empty graph.
\\
(iv) 
Here  we consider $(\sigma,\tau)$-translation schemes for sum-like operations,  with $\sigma =  \tau_{graphs} \cup \{a\}$.
Hence (*) from the proof of (iii)  is not true anymore because now $P_{m+n+1}$ can be obtained from
$P_n$ and $P_m$ with the $a$ being an end vertex, using $\sqcup^1$.
So we apply \autoref{le:periodic}.
\\
(v) Connection operations of two large enough cliques still produce connected graphs, but never form a clique.
\end{proof}

\subsection{Properties of Finite S-rank and Finite P-rank}

Let 
$\mathcal{S}(\tau)$  and
$\mathcal{P}(\tau)$ 
denote the collection of all $\tau$-properties of finite S-rank and finite P-rank respectively,
and let $\mathcal{S} = \bigcup_{\tau} \mathcal{S}(\tau)$ and
$\mathcal{P} = \bigcup_{\tau} \mathcal{P}(\tau)$. 

\begin{thm}
\label{th:nadia}
$\mathcal{S}$  and
$\mathcal{P}$  and
are abstract Lindstr\"om logics which have finite S-rank  and finite  P-rank, respectively.
\end{thm}
\begin{proof}[Sketch of proof:]
One first gives 
$\mathcal{S}$  and
$\mathcal{P}$ 
a canonical syntax as described in \cite{ar:Lindstr2},\cite{JAM-phd}.
The proof then is a tedious induction which will be published elsewhere.
\end{proof}

It is unclear whether the abstract Lindstr\"om logic $\mathcal{S}$ goes beyond $\CMSOL$.
As of now, we were unable to find a $\tau$-property which has finite S-rank, but is not definable in $\CMSOL$.

\begin{oproblem}
\ 
\begin{enumerate}
\item
Is every $\tau$-property with finite S-rank definable in $\CMSOL(\tau)$?
\item
Is every $\tau$-property with finite P-rank definable in $\CFOL(\tau)$?
\end{enumerate}
\end{oproblem}

It seems to us that the same can be shown for connection-like operations, but we have not yet checked the details.

\section{Hankel matrices and the Feferman-Vaught theorem}
\label{se:fv}
\subsection{The FV-property}
In this section we look at nice Lindstr\"om logics with a fixed quantifier rank.
We use it to derive from the classical Feferman-Vaught theorem an abstract version involving the quantifier rank.
This differs from the treatment in \cite[Chapter xviii]{bk:BF}.
Our purpose is to investigate the connection between Hankel matrices of finite rank and the
Feferman-Vaught Theorem  on finite structures in an abstract setting.

\begin{defi}
\label{def:fv}
Let $\mathcal{L}$ be a nice logic with quantifier rank $\rho$.
\begin{enumerate}
\item
We denote by $\cL(\tau)^q$ the set of $\cL(\tau)$-sentences $\phi$ (without free variables) with
$\rho(\phi) =q$.
\item
Two $\tau$-structures $\fA, \fB$ are $\mathcal{L}^q$ equivalent, $\fA \sim_{\mathcal{L}}^q \fB$,
if for every $ \phi \in \mathcal{L}(\tau)^q$
we have $\fA \models \phi$ iff $\fB \models \phi$.
\item
$\mathcal{L}$ has the FV-property for $\Box$ with respect to $\rho$ if for every $\phi \in \mathcal{L}(\tau)^q$
there are
$k=k(\phi) \in \N$, $\psi_1, \ldots, \psi_{k} \in \cL(\tau)^q$ and $B_{\phi} \in 2^{2k}$ such that
for all $\tau$-structures $\fA,\fB$ we have that
$$
\fA \Box \fB \models \phi
$$
iff
$$
B_{\phi}(\psi_1^A, \ldots \psi_k^A,\psi_1^B, \ldots \psi_k^B) =1
$$
where for $1 \leq j \leq k$ we have $\psi_j^A =1$ iff $\fA \models \psi_j$
and $\psi_j^B =1$ iff $\fB \models \psi_j$.
\item
$\Box$ is $\mathcal{L}$-smooth with respect to $\rho$ if for every two pairs of $\tau$-structures
$\fA_1, \fA_2, \fB_1, \fB_2$  with 
$\fA_1 \sim_{\mathcal{L}}^q \fA_2$
and
$\fB_1 \sim_{\mathcal{L}}^q \fB_2$
we also have
$\fA_1 \Box \fB_1 \sim_{\mathcal{L}}^q  \fA_2 \Box \fB_2$.
\end{enumerate}
If $\rho$ is clear from the context we omit it.
\end{defi}
A close inspection of the classical proofs shows that the requirements concerning
the quantifier rank are satisfied in the following cases.
%------------------------------------------------------------
\vskip1cm
\begin{examples}
\ 
\begin{itemize}
\item
(\cite{ar:FefermanVaught}):
$\FOL$ has the FV-property for all product-like  and connection-like operations $\Box$.
\item
(\cite{ar:KotekMakowsky-LMCS2014}):
$\CFOL$ with quantifier rank $\rho_2$ has the FV-property for all product-like and connection-like operations $\Box$.
\item
(\cite{ar:gurevich79},\cite{pr:laeuchli68},\cite{ar:shelah75}):
$\MSOL$ has the FV-property for all sum-like  and connection-like operations $\Box$.
\item
(\cite{ar:courcelle90}):
$\CMSOL$ with quantifier rank $\rho_2$ has the FV-property for all sum-like and connection-like operations $\Box$.
\end{itemize}
\end{examples}
%------------------------------------------------------------
\subsection{The FV-property and Finite Rank}
\begin{defi}
Let $\cL$ be a nice logic.
\begin{enumerate}
\item
Let $\Box$ be a binary operation on $\tau$-structures.
$\cL$ is $\Box$-closed if 
all the equivalence classes of $\equiv_{\phi, \Box}$ are definable
in $\cL(\tau)$.
\item
$\cL$ is S-closed (P-closed, C-closed) if for every sum-like (product-like, connection-like) binary operation $\Box$ the logic
$\cL$ is $\Box$-closed.
\end{enumerate}
\end{defi}

\begin{prop}
\label{prop:fv-1}
Let $\mathcal{L}$ have the FV-property for $\Box$. 
\begin{enumerate}
\item
$\Box$ is $\mathcal{L}$-smooth.
\item
Let $\Phi$ be a $\tau$-property
definable by a formula $\phi \in \mathcal{L}(\tau)^q$.
Then each equivalence class $[\fA]_{\Phi,\Box}$ of $\equiv_{\Phi,\Box}$
is definable by a formula $\psi(\fA) \in  \mathcal{L}(\tau)^q$.
\item
If $\cL$ has the FV-property for all sum-like (product-like) operations
then $\cL$ is S-closed (P-closed).
\end{enumerate}
\end{prop}
\begin{proof}[Sketch of proof]
(i) Follows because  for $i =1,2$, the truth value of
$\fA_i \Box \fB_i \models \phi \in \mathcal{L}(\tau)^q$ depends only on
$B_{\phi}$, the Boolean function associated with the FV-property.

(ii) Fix a $\tau$-structure $\fA$.
We want to show that $[\fA]_{\Phi,\Box}$ is definable by some formula $\psi(\fA) \in \mathcal{L}(\tau)^q$.

$\fB \in [\fA]_{\Box,\Phi}$ iff for all $\fC$, 
$\fA \Box \fC \in \Phi$ iff
$\fB \Box \fC \in \Phi$. 
\\
We have, using $B_{\phi}$, that
$$
\fA \equiv_{\Phi,\Box} \fB 
$$
iff for all $\fC$,
\begin{equation}
B_{\phi}( \psi_1^A, \ldots , \psi_k^A, \psi_1^C, \ldots , \psi_k^C)
=
B_{\phi}( \psi_1^B, \ldots , \psi_k^B, \psi_1^C, \ldots , \psi_k^C)
\end{equation}
iff
$\forall X_1, \ldots , X_k \in \{0,1\}$,
\begin{equation}
\label{eq:1}
B_{\phi}( \psi_1^A, \ldots , \psi_k^A, X_1, \ldots, X_k)
=
B_{\phi}( \psi_1^B, \ldots , \psi_k^B, X_1, \ldots, X_k)
\end{equation}
where $\psi_i^A$, 
$\psi_i^B$ and  $\psi_i^C$ are as in Definition \ref{def:fv}(iii).
Equation (\ref{eq:1}) can be expressed by a formula $\psi(\fA) \in \mathcal{L}(\tau)^q$.

(iii) Follows from (ii).
\end{proof}
%-----------------------------------------------------------------

By analyzing the proof in \cite{ar:GodlinKotekMakowsky08}, one can prove:
\begin{thm}
\label{th:fv-2}
Let $\cL$ be a nice  Lindstr\"om logic with quantifier rank $\rho$ and $\Box$ be a binary  operation on $\tau$-structures.
If $\Box$ is $\cL$-smooth with respect to $\rho$,
then every $\cL$-definable $\tau$-property $\Phi$ has finite $\Box$-rank.
\end{thm}
\begin{proof}[Sketch of proof]
Let $\Phi$ be definable by $\phi$ with quantifier rank $\rho(\phi)=q$.
Now let $\phi_i: i \leq \alpha \in \N$ be an enumeration of maximally consistent $\cL(\tau)^q$-sentences (aka Hintikka sentences).
By our assumption $\rho$ is nice, so this is a finite set.
Furthermore $\phi$ is logically equivalent to a disjunction  $\bigvee_{i \in I} \phi_i$ with $I \subseteq [\alpha]$,
any every $\tau$-structure satisfies exactly one $\phi_i$. 
\\
Now we use the  smoothness of $\Box$. 
If $\fA,\fB$ are two $\tau$-structures satisfying the same $\phi_i$, then their rows (columns) in $H(\Phi,\Box)$ are identical.
Hence the rank of $H(\Phi,\Box)$ is at most $\alpha$, or $\alpha +1$ when empty $\tau$-structures are allowed.
\end{proof}

Combining \autoref{th:fv-2} with \autoref{prop:fv-1}(i) we get:
\begin{coro}
Let $\mathcal{L}$ be a nice Lindstr\"om logic which has the FV-property
for the binary operation $\Box$, and
let $\Phi$ be definable in $\mathcal{L}$.
Then $r(\Phi, \Box)$  is finite.
\end{coro}

%---------------------------------------------------------------------
\ifskip
\else
\begin{prop}
\label{prop:local}
Let $\mathcal{L}$ be a nice logic with quantifier rank $\rho$. 
\begin{enumerate}
\item
Let $\Box$ be a fixed operation on $\tau$-structure such that  
\begin{enumerate}
\item
$\Box$ is associative, and
\item
for every $\phi \in \mathcal{L}(\tau)$ the rank of $H(\phi, \Box)$ is finite, and
\item
all equivalence classes of $\equiv_{\phi,\Box}$ are definable by formulas of $\mathcal{L}$
with quantifier rank $ \leq qr(\phi)$.
\end{enumerate}
Then $\mathcal{L}$ has the FV-property for $\Box$.
\item
If $\cL$ 
has finite S-rank then the S-closure of $\cL$ has the FV-property for all sum-like operations.
\end{enumerate}
\end{prop}
\fi
%---------------------------------------------------------------------

\begin{prop}
\label{prop:local}
Let $\mathcal{L}$ be a nice logic with quantifier rank $\rho$ 
and $\Box$ be a fixed operation on $\tau$-structure, which is associative.
Assume further that for every $\phi \in \mathcal{L}(\tau)$, 
\begin{enumerate}
\item
the rank of $H(\phi, \Box)$ is finite, and
\item
all equivalence classes of $\equiv_{\phi,\Box}$ are definable with formulas of $\mathcal{L}$
with quantifier rank $ \leq qr(\phi)$.
\end{enumerate}
Then $\mathcal{L}$ has the FV-property for $\Box$.
\end{prop}

We have now shown that $\mathcal{L}$ having the FV-property for $\Box$ implies that $\Box$ is $\mathcal{L}$-smooth,
and that smoothness implies finite rank, or equivalently, finite index.

In fact we have:
\begin{thm}
\label{th:fv-main}
Let $\cL$ be a nice S-closed logic and let $\Box_1$ be a sum-like operation.
Then the following are equivalent:
\begin{enumerate}
\item
$\cL$ has the FV-property for every sum-like operation $\Box$.
\item
$\Box_1$ is $\cL$-smooth. 
\item
For all $\phi \in \cL(\tau)$ and every sum-like $\Box$, the $\Box$-rank of $\phi$ is finite.
\item
For all $\phi \in \cL(\tau)$ and every sum-like $\Box$, the index of
$\equiv_{\phi, \Box}$ is finite.
\end{enumerate}
The same holds if we replace S-closed and  sum-like by P-closed and product-like.
\end{thm}

\begin{proof}
(i) implies (ii) is \autoref{prop:fv-1}.
\\
(ii) implies (iii) is \autoref{th:fv-2}.
\\
(iii) is equivalent to (iv) by  \autoref{prop:index}.
\\
Finally, (iii) implies (i) is  \autoref{prop:local}.
\end{proof}

\section{The S-closure of a nice logic}
\label{se:sclosure}

Let $\cL$ be a nice logic of finite S-rank with quantifier rank $\rho$. 
We define
$Cl_S(\cL)$ to be the smallest Lindstr\"om logic
such that for all sum-like 
$$
\Box: \Str(\sigma) \times \Str(\sigma) \rightarrow \Str(\tau)
$$ 
and all $Cl_S(\cL)$-definable $\tau$-properties $\Phi$,
all the equivalence classes of $\equiv_{\Phi, \Box}$ are also definable in $Cl_S(\cL)$.
This gives us a Lindstr\"om logic which is S-closed. 
However, in order to be a nice logic, we have to extend $\rho$ to $\rho'$ in such a way that ensures it is still nice.

We proceed inductively. 
Recall that there are only finitely many sum-like operations $\Box$ for fixed $\sigma$ and $\tau$.
Let $\ell(\sigma) = \sum_{R \in \sigma} r(R)+1$ where $R$ is a relation symbol of arity $r(R)$ or a constant symbol of arity $0$.
Two vocabularies are {\em similar} if they have the same number of symbols of the same arity.
The effect of a sum-like operation only depends on the similarity type of $\sigma$ and $\tau$.
Hence for fixed $\ell(\sigma)$ and $\ell(\tau)$, there are only finitely many sum-like operations.

A typical step in the induction is as follows.

Given $\cL$ and $\phi \in \cL(\tau)^{\rho(\phi)}$ and a sum-like 
$ \Box: \Str(\sigma) \times \Str(\sigma) \rightarrow \Str(\tau) $, 
there are only finitely many equivalence classes of $\equiv_{\phi, \Box}$.
Let $E_i= E(\phi, \Box)_i$ with $i \leq \alpha=\alpha(\phi, \Box)$ be a list of these equivalence classes.  

We form $\cL'$ with quantifier rank $\rho'$ as follows:
If $E_i$ is not definable in $\cL(\sigma)$ then we add it to $\cL$ using a Lindstr\"om quantifier 
with quantifier rank $\rho'(E_i)=\rho(\phi)+\ell(\sigma) + \ell(\tau)$.

$\cL'$ is a Lindstr\"om logic. We have to show that $\rho'$ is nice, i.e., for fixed $q$ and fixed number of free variables,
$\cL'(\tau)^q$ is finite up to logical equivalence. 
This follows from the fact that we only added finitely many Lindstr\"om quantifiers and that
for all $\phi \in \cL$ we have that $\rho'(\phi) = \rho(\phi)$.

For our induction we start with $\cL_0 = \cL$. 
$\cL_1$ is obtained by doing the typical step for each $\phi \in \cL_0$ and each sum-like $\Box$.
$\rho_1$ is the union of all quantifier rank functions of the previous steps. 
We still have iterate this process by defining $\cL_j$ and $\rho_j$ and take the limit.

We finally get:
\begin{thm}
Let $\cL$ be nice with quantifier rank $\rho$ and of finite S-rank.
Then $Cl_S(\cL)$ with quantifier rank $\rho'$ is nice and has the FV-property for all sum-like operations.
\end{thm}
The details will be published in \cite{up:LabaiMakowskyFV}.

\section{Conclusions and open problems}
\label{se:conclu}
At the beginning of this paper
we asked whether one can characterize logics over finite structures which
satisfy the Feferman-Vaught Theorem for the disjoint union, or more generally, for sum-like and product-like
operations on $\tau$-structures.
The purpose of this paper was to investigate new directions to attack this problem,
specifically by relating the Feferman-Vaught Theorem to Hankel matrices of finite rank.
\autoref{th:fv-main} describes their exact relationship.

We also investigated under which conditions one can construct logics satisfying the Feferman-Vaught Theorem.
 \autoref{th:1} shows that there are uncountably many $\tau$-properties which have
finite rank Hankel matrices for specific sum-like operations.
\autoref{th:nadia} shows the existence of maximal
Lindstr\"om logics $\mathcal{S}$ and $\mathcal{P}$ where all
their definable $\tau$-properties have finite rank for 
all sum-like, respectively product-like, operations.
However, we have no explicit description of these maximal logics.

\begin{oproblem}
\ 
\begin{enumerate}
\item
Is every $\tau$-property with finite P-rank  (or both finite P-rank and finite C-rank)  definable
in $\CFOL$?
\item
Is every $\tau$-property with finite S-rank (finite C-rank)  definable
in $\CMSOL$? 
\end{enumerate}
\end{oproblem}
In case the answers to the above are negative, we can ask:
\begin{oproblem}
\ 
\begin{enumerate}
\item
How many $\tau$-properties are there with finite S-rank (P-rank, C-rank)?
\item
Is there a nice Gurevich logic where all the 
$\tau$-properties in $\mathcal{S}$ are definable?
\end{enumerate}
\end{oproblem}
In \cite[Section 7, Conjecture 2]{ar:MakowskyTARSKI} it is conjectured 
that there are continuum 
many nice Gurevich logics with the FV-property for the disjoint union.
Adding $C_A$ or $P_A$ 
from  \autoref{th:1}   
for fixed $A \subseteq \N$ 
as Lindstr\"om quantifiers to $\FOL$ together with all the 
equivalence classes of 
$\equiv_{C_A, \sqcup}$ 
or
$\equiv_{P_A, \sqcup}$ 
gives us a nice Lindstr\"om logic. However, the  definable $\tau_{graph}$-property
that the complement of a graph $G$ is in $C_A$ has infinite $\sqcup$-rank,
see  \autoref{th:1}(ii).

\begin{oproblem}
How many different nice Gurevich logics with the FV-property for the disjoint union are there?
\end{oproblem}

A similar analysis for connection-like operations will be developed in \cite{up:LabaiMakowskyFV}.
{\small
\subsubsection*{Acknowledgments}
We would like to thank T. Kotek for letting us use his example, and for valuable discussions.

}


\begin{thebibliography}{50}

\bibitem{bk:BF}
J.~Barwise and S.~Feferman, editors.
\newblock {\em Model-Theoretic Logics}.
\newblock Perspectives in Mathematical Logic. Springer Verlag, 1985.

\bibitem{ar:Barwise74}
K~Jon Barwise.
\newblock Axioms for abstract model theory.
\newblock {\em Annals of Mathematical Logic}, 7(2):221--265, 1974.

\bibitem{pr:BlatterSpecker81}
C.~Blatter and E.~Specker.
\newblock Le nombre de structures finies d'une th'eorie \`a charact\`ere fin.
\newblock {\em Sciences Math\'ematiques, Fonds Nationale de la recherche
  Scientifique, Bruxelles}, pages 41--44, 1981.

\bibitem{ar:CarlylePaz1971}
J.W. Carlyle and A.~Paz.
\newblock Realizations by stochastic finite automata.
\newblock {\em J. Comp. Syst. Sc.}, 5:26--40, 1971.

\bibitem{pr:ChandraLewisM}
A.~Chandra, H.~Lewis, and J.A. Makowsky.
\newblock Embedded implicational dependencies and their implication problem.
\newblock In {\em ACM Symposium on the Theory of Computing 1981}, pages
  342--354. ACM, 1981.

\bibitem{pr:ChandraLewisM80}
A.~K. Chandra, H.~R. Lewis, and J.~A. Makowsky.
\newblock Embedded implicational dependencies and their inference problem.
\newblock In {\em XP1 Workshop on Database Theory}, 1980.

\bibitem{bk:CK}
C.C. Chang and H.J. Keisler.
\newblock {\em Model Theory}.
\newblock Studies in Logic, vol 73. North--Holland, 3rd edition, 1990.

\bibitem{ar:Compton83}
K.~J. Compton.
\newblock Some useful preservation theorems.
\newblock {\em J. Symb. Log.}, 48(2):427--440, 1983.

\bibitem{ar:courcelle90}
B.~Courcelle.
\newblock The monadic second--order logic of graphs {I}: Recognizable sets of
  finite graphs.
\newblock {\em Information and Computation}, 85:12--75, 1990.

\bibitem{durand2012fifty}
Arnaud Durand, Neil~D Jones, Johann~A Makowsky, and Malika More.
\newblock Fifty years of the spectrum problem: survey and new results.
\newblock {\em Bulletin of Symbolic Logic}, 18(04):505--553, 2012.

\bibitem{bk:EFT94}
H.-D. Ebbinghaus, J.~Flum, and W.~Thomas.
\newblock {\em Mathematical Logic, 2nd edition}.
\newblock Undergraduate Texts in Mathematics. Springer-Verlag, 1994.

\bibitem{ar:Feferman68}
S.~Feferman.
\newblock Persistent and invariant formulas for outer extensions.
\newblock {\em Compositio Mathematicae}, 20:29--52, 1968.

\bibitem{ar:Feferman74b}
S.~Feferman.
\newblock Two notes on abstract model theory, {I}: Properties invariant on the
  range of definable relations between strcutures.
\newblock {\em Fundamenta Mathematicae}, 82:153--165, 1974.

\bibitem{ar:FefermanKreisel66}
S.~Feferman and G.~Kreisel.
\newblock Persistent and invariant formulas relative to theories of higher
  order.
\newblock {\em Bulletin of the American Mathematical Society}, 72:480--485,
  1966.

\bibitem{ar:FV}
S.~Feferman and R.~Vaught.
\newblock The first order properties of algebraic systems.
\newblock {\em Fundamenta Mathematicae}, 47:57--103, 1959.

\bibitem{ar:FefermanVaught}
S.~Feferman and R.~Vaught.
\newblock The first order properties of products of algebraic systems.
\newblock {\em Fundamenta Mathematicae}, 47:57--103, 1959.

\bibitem{ar:FischerKotekMakowsky11}
E.~Fischer, T.~Kotek, and J.A. Makowsky.
\newblock Application of logic to combinatorial sequences and their recurrence
  relations.
\newblock In M.~Grohe and J.A. Makowsky, editors, {\em Model Theoretic Methods
  in Finite Combinatorics}, volume 558 of {\em Contemporary Mathematics}, pages
  1--42. American Mathematical Society, 2011.

\bibitem{ar:FreedmanLovaszSchrijver07}
M.~Freedman, L{\'a}szl{\'o} Lov{\'a}sz, and A.~Schrijver.
\newblock Reflection positivity, rank connectivity, and homomorphisms of
  graphs.
\newblock {\em Journal of AMS}, 20:37--51, 2007.

\bibitem{ar:Gessel84}
I.~Gessel.
\newblock Combinatorial proofs of congruences.
\newblock In D.M. Jackson and S.A. Vanstone, editors, {\em Enumeration and
  design}, pages 157--197. Academic Press, 1984.

\bibitem{ar:GodlinKotekMakowsky08}
B.~Godlin, T.~Kotek, and J.A. Makowsky.
\newblock Evaluation of graph polynomials.
\newblock In {\em 34th International Workshop on Graph-Theoretic Concepts in
  Computer Science, WG08}, volume 5344 of {\em Lecture Notes in Computer
  Science}, pages 183--194, 2008.

\bibitem{Graedel91b}
E.~Gr{\"a}del.
\newblock The expressive power of second order horn logic.
\newblock In {\em Proceedings of 8th Symposium on Theoretical Aspects of
  Computer Science STACS `91, Hamburg 1991}, volume 480 of {\em LNCS}, pages
  466--477. Springer-Verlag, 1991.

\bibitem{ar:gurevich79}
Y.~Gurevich.
\newblock Modest theory of short chains, {I}.
\newblock {\em Journal of Symbolic Logic}, 44:481--490, 1979.

\bibitem{bk:gurevich}
Y.~Gurevich.
\newblock Logic and the challenge of computer science.
\newblock In E.~B\"{o}rger, editor, {\em Trends in Theoretical Computer
  Science}, Principles of Computer Science Series, chapter~1. Computer Science
  Press, 1988.

\bibitem{ar:IMM1}
N.~Immerman.
\newblock Languages that capture complexity classes.
\newblock {\em SIAM Journal on Computing}, 16(4):760--778, Aug 1987.

\bibitem{ar:KotekMakowsky-LMCS2014}
T.~Kotek and J.A. Makowsky.
\newblock Connection matrices and the definability of graph parameters.
\newblock {\em Logical Methods in Computer Science}, 10(4), 2014.

\bibitem{ar:Kreisel68}
G.~Kreisel.
\newblock Choice of infinitary languages by means of definability criteria;
  generalized recursion theory.
\newblock In J.~Barwise, editor, {\em The Syntax and Semantics of Infinitary
  Languages}, volume~72 of {\em Springer Lecture Notes in Mathematics}, pages
  139--151. Springer, 1968.

\bibitem{up:LabaiMakowskyFV}
N.~Labai and J.A. Makowsky.
\newblock The Feferman-Vaught theorem and finite Hankel rank.
\newblock In preparation, 2015.

\bibitem{ar:LabaiMakowskyFPSAC}
N.~Labai and J.~Makowsky.
\newblock Tropical graph parameters.
\newblock {\em DMTCS Proceedings of 26th International Conference on Formal
Power Series and Algebraic Combinatorics (FPSAC)}, (01):357--368, 2014.


\bibitem{pr:laeuchli68}
H.~L{\"a}uchli.
\newblock A decision procedure for the weak second order theory of linear
  order.
\newblock In {\em Logic Colloquium '66}, pages 189--197. North Holland, 1968.

\bibitem{ar:Lindstr1}
P.~Lindstr\"{o}m.
\newblock First order predicate logic with generalized quantifiers.
\newblock {\em Theoria}, 32:186--195, 1966.

\bibitem{ar:Lindstr2}
P.~Lindstr\"om.
\newblock On extensions of elementary logic.
\newblock {\em Theoria}, 35:1--11, 1969.

\bibitem{ar:Lovasz07}
L.~Lov\'asz.
\newblock Connection matrics.
\newblock In G.~Grimmet and C.~McDiarmid, editors, {\em Combinatorics,
  Complexity and Chance, A Tribute to Dominic Welsh}, pages 179--190. Oxford
  University Press, 2007.

\bibitem{bk:Lovasz-hom}
L.~Lov{\'a}sz.
\newblock {\em Large Networks and Graph Limits}, volume~60 of {\em Colloquium
  Publications}.
\newblock AMS, 2012.

\bibitem{ar:MMahr}
B.~Mahr and J.A. Makowsky.
\newblock Characterizing specification languages which admit initial semantics.
\newblock {\em Theoretical Computer Science}, 31:49--60, 1984.

\bibitem{JAM-phd}
J.~A. Makowsky.
\newblock {\em $\Delta$-logics and generalized quantifiers}.
\newblock PhD thesis, Department of Mathematics, ETH-Zurich, Switzerland, 1974.

\bibitem{ar:JAMhorn}
J.~A. Makowsky.
\newblock Why {H}orn formulas matter for computer science: Initial structures
  and generic examples.
\newblock {\em Journal of Computer and System Sciences}, 34.2/3:266--292, 1987.

\bibitem{pr:JAMmth1}
J.A. Makowsky.
\newblock Model theoretic issues in theoretical computer science, part {I}:
  Relational databases and abstract data types.
\newblock In G.~Lolli et~al., editor, {\em Logic Colloquium '82}, Studies in
  Logic, pages 303--343. North Holland, 1984.

\bibitem{bk:BFxviii}
J.A. Makowsky.
\newblock Compactness, embeddings and definability.
\newblock In J.~Barwise and S.~Feferman, editors, {\em Model-Theoretic Logics},
  Perspectives in Mathematical Logic, chapter~18. Springer Verlag, 1985.

\bibitem{pr:JAM-kgc97}
J.A. Makowsky.
\newblock Invariant definability.
\newblock In {\em Computational Logic and Proof Theory}, volume 1289 of {\em
  Lecture Notes in Computer Science}, pages 186--202. Springer Verlag, 1997.

\bibitem{ar:MakowskyTARSKI}
J.A. Makowsky.
\newblock Algorithmic uses of the {F}eferman-{V}aught theorem.
\newblock {\em Annals of Pure and Applied Logic}, 126.1-3:159--213, 2004.

\bibitem{ar:MPlcc94}
J.A. Makowsky and Y.~Pnueli.
\newblock Logics capturing oracle complexity classes uniformly.
\newblock In {\em Logic and Computational Complexity (LCC'94)}, volume 960 of
  {\em Lecture Notes in Computer Science}, pages 463--479. Springer Verlag,
  1995.

\bibitem{ar:MPcsl93}
J.A. Makowsky and Y.B. Pnueli.
\newblock Oracles and quantifiers.
\newblock In {\em CSL'93}, volume 832 of {\em Lecture Notes in Computer
  Science}, pages 189--222. Springer, 1994.

\bibitem{ar:MVardi}
J.A. Makowsky and M.~Vardi.
\newblock On the expressive power of data dependencies.
\newblock {\em Acta Informatica}, 23.3:231--244, 1986.

\bibitem{th:Ravve95}
E.~Ravve.
\newblock Model checking for various notions of products.
\newblock Master's thesis, Thesis, Department of Computer Science,
  Technion--Israel Institute of Technology, 1995.

\bibitem{ar:shelah75}
S.~Shelah.
\newblock The monadic theory of order.
\newblock {\em Annals of Mathematics}, 102:379--419, 1975.

\bibitem{ar:Specker88}
E.~Specker.
\newblock Application of logic and combinatorics to enumeration problems.
\newblock In E.~B{\"o}rger, editor, {\em Trends in Theoretical Computer
  Science}, pages 141--169. Computer Science Press, 1988.
\newblock Reprinted in: Ernst Specker, Selecta, Birkh\"auser 1990, pp. 324-350.

\bibitem{pr:vardi82}
M.~Vardi.
\newblock The complexity of relational query languages.
\newblock In {\em STOC'82}, pages 137--146. ACM, 1982.

\end{thebibliography}
\end{document}